\documentclass[a4paper]{amsart}

\usepackage[utf8]{inputenc}
\usepackage[T1]{fontenc}
\usepackage[english]{babel}
\usepackage{color}
\usepackage{amsmath} 
\usepackage{amssymb} 
\usepackage{amsthm}
\usepackage{graphicx}
\usepackage[colorlinks=true, linkcolor=blue]{hyperref}
\usepackage{cancel}

\usepackage{lmodern}
\usepackage{microtype}

\theoremstyle{plain}
\newtheorem{thm}{Theorem}[section]
\newtheorem*{thm*}{Theorem}
\newtheorem{prop}{Proposition}[section] 
\newtheorem{lem}{Lemma}[section] 
\newtheorem{cor}{Corollary}[section]
\newtheorem*{cor*}{Corollary}

\newtheorem{defi}{Definition}[section]

\newtheorem{rem}{Remark}

\newcommand {\R} {\mathbb{R}} \newcommand {\Z} {\mathbb{Z}}
\newcommand {\T} {\mathbb{T}} 

\newcommand {\p} {\partial}
\newcommand {\dt} {\partial_t}

\begin{document}
\title[Linear inviscid damping without scattering]{On the forced Euler and
  Navier-Stokes equations: Linear damping and modified scattering}
\begin{abstract}
  We study the asymptotic behavior of the forced linear Euler and nonlinear
  Navier-Stokes equations close to Couette flow on $\T \times I$. As our main result we show that for smooth
  time-periodic forcing linear inviscid damping persists, i.e. the velocity
  field (weakly) asymptotically converges.
  However, stability and scattering to the transport problem fail in $H^{s}, s>-1$.
  We further show that this behavior is consistent with the nonlinear Euler
  equations and that a similar result also holds for the nonlinear Navier-Stokes
  equations.
  Hence, these results provide an indication that nonlinear inviscid damping may still hold in Sobolev
  regularity in the above sense despite the Gevrey regularity instability
  results of \cite{dengmasmoudi2018}.
\end{abstract}
\author{Christian Zillinger}
\address{
Department of Mathematics,
University of Southern California,
3620 S. Vermont Avenue,
Los Angeles, CA 90089-2532, US}
\email{zillinge@usc.edu}
\maketitle
\tableofcontents
\section{Introduction}
\label{sec:introduction}
The problems of inviscid damping and enhanced dissipation are classical, going
back to the works of \cite{orr1907stability} and Rayleigh \cite{Rayleigh} around
1900, who studied the linearized Euler and Navier-Stokes equations around
Couette flow, $v=(U(y),0)$.
In that case, the linearized inviscid problem in terms of the vorticity is given by
\begin{align*}
  \partial_t \omega + y \p_x \omega =0 
\end{align*}
and explicitly solvable by the method of characteristics as $\omega(t,x,y)=\omega_0(x-ty,y)$.
In particular, one observes that the vorticity converges weakly (but not
strongly) in $L^2$ to its $x$ average and as a consequence, the corresponding
velocity field converges strongly to a shear flow. This in view of the
conservation law structure of Euler's equations at first very unexpected
stabilization mechanism is known as \emph{inviscid damping} in analogy to Landau
damping in plasma physics.
However, while the result for Couette flow is classical, until recently little
has been known about linearizations around other profiles or the associated nonlinear problem.
Following the seminal works of Mouhot and Villani on Landau damping for the
Vlasov-Poisson equation, \cite{Villani_long}, \cite{Villani_short},
\cite{Villani_script}, \cite{Landau}, \cite{bedrossian2016landau} these problems have attracted much renewed
interest. For a more extensive discussion of the literature we refer to
\cite{bedrossian2016enhanced} and briefly mention the following works:
\begin{itemize}
\item The problem of nonlinear inviscid damping and enhanced dissipation for
  Couette flow in an infinite periodic channel has been studied in series of
  works by Bedrossian, Masmoudi, Vicol, Wang and others \cite{bedrossian2016enhanced}, \cite{bedrossian2016sobolev},
  \cite{bedrossian2015inviscid}, \cite{bedrossian2015dynamics},
  \cite{bedrossian2015dynamics2}, \cite{bedrossian2016enhanced},
  \cite{bedrossian2013asymptotic}, \cite{bedrossian2013inviscid}. Here, in the
  inviscid setting Gevrey regularity has been shown to be necessary to control
  the effect of nonlinear resonances, also called \emph{echos}, \cite{dengmasmoudi2018}.
\item Concerning the linearized problem around profiles different from Couette
  efforts have focused on establishing inviscid damping and enhanced dissipation for more general and
  degenerate classes of profiles.
  Here, the works of the author and Coti Zelati, \cite{Zill6}, \cite{Zill3},
  \cite{Zill5}, \cite{CZZ17} rely on a dispersive approach using multiplier and
  decompositions. Using a spectral approach Wei, Zhang and Zhao have been
  able to establish linear inviscid damping for different, more general classes
  of profiles, \cite{Zhang2015inviscid}, \cite{wei2018linear}, \cite{wei2018transition}, \cite{wei2017linear},
  \cite{stepin1995nonself}.
\item When considering the setting of a finite periodic channel, $\T \times
  [0,1]$, or other domains with boundaries one additionally encounters
  instabilities due to boundary effects, which generally restrict to working with
  fractional Sobolev spaces $H^s, s<\frac{3}{2}$ or weighted $H^2$ spaces,
  \cite{Zill5}, \cite{Zill6}, \cite{Zhang2015inviscid}, \cite{CZZ17}.
  This is of particular interest in view of the much higher regularity requirements
  of present results on the nonlinear problem and known lower bounds on the required
  regularity \cite{Lin-Zeng}.
  Here, recently Ionescu and Jia \cite{Ionescu2018inv} have been able to show
  that for Gevrey regular vorticity compactly supported away from the boundary
  these blow-up mechanisms can be avoided and nonlinear inviscid damping holds.
\item As a related problem damping, mixing and enhanced dissipation for passive
  scalar problems are an active area of research \cite{zelati2018relation},
  \cite{alberti2014exponential}, \cite{crippa2017cellular}, \cite{Crippa17}, \cite{zillinger2018geometric}.
\end{itemize}
Following the results of Bedrossian-Masmoudi \cite{bedrossian2015inviscid} on
inviscid damping in Gevrey regularity, it has been shown by Deng-Masmoudi in
\cite{dengmasmoudi2018} that Gevrey regularity is necessary for (asymptotic)
stability with respect to the linearized dynamics in the setting of an unbounded
channel.

However, as also noted there, it remains an open question whether inviscid damping, i.e. convergence in
$H^{-1}$, fails otherwise. This question is particularly relevant for the
setting of a finite channel, where generally boundary instabilities yield asymptotic
blow-up as $t\rightarrow \infty$ in $W^{1,\infty}$ and consequently in relatively low Sobolev norms \cite{Zill5}.

In this work we make a first modest step in the direction of addressing this
question and consider the behavior of the forced equations for shear flows close
to Couette flow in a periodic channel:
\begin{align}
  \label{eq:1}
  \begin{split}
    \dt v + v \cdot \nabla v &= \nabla p +F + \nu \Delta v, \\
    \nabla \cdot v &=0.
  \end{split}
\end{align}
The velocity field of the fluid is denoted by $v \in \R^2$ and $F \in \R^{2}$ is a
given force field. The pressure $p$ can be interpreted as a Lagrange multiplier
ensuring the incompressibility, i.e. $p$ is determined by solving
\begin{align*}
  \Delta p = \nabla \cdot (v \cdot \nabla v)- \nabla \cdot F.
\end{align*}
In this article we consider
\begin{itemize}
\item the linearized forced Euler equations on $\T \times I$ around shear
  profiles $U(y)$ satisfying linear inviscid damping (c.f. properties
  \eqref{eq:9} to \eqref{eq:11} in Section \ref{sec:linEuler}),
\item the associated consistency problem and
\item the nonlinear Navier-Stokes equations on $\T \times \R$.
\end{itemize}
The choice of a bounded interval $I$ for the inviscid setting here allows us to
restrict to the case when  $U(y)$ is bounded above and below and so that we may
restrict to studying unweighted spaces. However, as a severe drawback the comparably low Sobolev
regularity available in this setting, makes the consistency problem of Section
\ref{sec:linEuler} very challenging.
In Section \ref{sec:simple}, when considering the Navier-Stokes equations, we instead
restrict to the case without boundary, $\T \times \R$.
\\

We study stability in Sobolev regularity, where we consider smooth deterministic
forcing of the following types, where $f= \nabla \times F$ (or the perturbation thereof):
\begin{itemize}
\item $f \in H^{1}_{loc}(\R_{+}, H^{1}(\T \times I))$ is of ``stationary type''
  if it is periodic in $t$ with period $T>0$. 
\item $f \in H^{1}_{loc}(\R_{+}, H^{1}(\T \times I))$ is ``resonant'' if
  $f(t,x-tU(y),y)$ is periodic in $t$ with period $T>0$.
\end{itemize}
The types of forcing considered here are intended to mimic the effect of echoes
and thus provide insights into what kind of (in)stability results and whether
(modified) scattering results can be expected.
In particular, we show that a control in just $H^{1}_{loc}(\R_{+}, H^{1}(\T
\times I))$ is ``too rough'' to capture cancellation behavior and that inviscid
damping may persists despite instability.

\subsection{Main results}
\label{sec:main}

As our first main result, we show that for the forced linearized inviscid problem
\begin{align}
  \label{eq:lin}
  \begin{split}
    \dt \omega + U(y)\p_{x} \omega - U''(y)\p_{x}\phi&= f, \\
    \Delta \phi &= \omega,
  \end{split}
\end{align}
with the above types of forcing both linear inviscid damping and algebraic
instability in Sobolev regularity may hold at the same time.
Furthermore, this behavior is consistent with the nonlinear equations, where
both types of forcing naturally appear.

Our main results for the forced linearized inviscid problem and the associated
consistency problem are summarized in the following theorem.
\begin{thm}
  \label{thm:main}
  Let $U$ be a flow profile that is bounded above and below on the interval $I \subset \R$
  and such that linear inviscid damping holds in the sense that conditions
  \eqref{eq:9} to \eqref{eq:11} are satisfied (c.f. Section \ref{sec:linEuler}).
  Suppose further that $f \in H^{1}_{loc}(\R_{+}, H^{1}(\T \times I))$ is
  periodic in $t$ with period $T>0$ with vanishing average in $x$,  $\int f(t,x,y)dx=0$.
  Then we obtain the following results when $f$ is of resonant type:
  \begin{enumerate}
  \item The evolution for $W(t,x,y):= \omega(t,x+tU(y),y)$ is algebraically
    unstable in $H^{s}$ for any $s>-1$.
  \item Linear inviscid damping holds in a weak sense. That is, $v(t)$ is
    uniformly bounded in $L^{2}$ and there exists $v_{\infty} \in L^{2}$
    such that \[v(t) \rightharpoonup v_{\infty} \]
    as $t \rightarrow \infty$.
  \end{enumerate}
  If instead $f$ is of stationary type and non-degenerate (c.f. Section
  \ref{sec:linEuler}), then
\begin{enumerate}
\item The evolution for $W(t,x,y):= \omega(t,x+tU(y),y)$ is algebraically
    unstable in $H^{s}$ for any $s>0$. 
  \item While $\omega(t)$ is stable in $L^{2}$, neither $w(t,x,y)$ nor $W(t,x,y)$
    converge as $t \rightarrow \infty$.
    Instead there is a sum space decomposition
    \[\omega(t,x,y)= \omega_{1}(t,x-tU(y),y) + \omega_{2}(t,x,y),\]
    where both $\omega_{1}, \omega_{2}$ are stable in $L^{2}$ and converge as $t
    \rightarrow \infty$.
  \item The evolution of $\omega$ is asymptotically stable in $H^{s}$ for any
    $-1\leq s<0$. In particular, linear inviscid damping holds.
  \end{enumerate}
  Let $\omega(t)$ denote the solution of the forced linearized Euler equations
  where $f$ is in the stationary case and assume that $\omega_0, \p_x \omega_0,
  g, \p_x \omega_0 \in H^{3/2}\cap W^{1,\infty}$.
  Then the Duhamel integral
    \begin{align*}
    \sigma(t,x+tU(y),y):= \int_0^t S(t,\tau)(\nabla^\bot \phi \cdot \nabla \omega)(\tau,x-\tau U(y), y) d\tau
  \end{align*}
  is uniformly bounded in $H^{-1}$ and weakly asymptotically convergent, but grows unbounded in $H^s$ for any $s>-1$.
\end{thm}

\begin{rem}
  In view of existing results for the unforced problem, the most important new
  phenomena and results here are:
  \begin{itemize}
  \item We give an explicit setting of the forced linearized problem, where linear
    inviscid damping holds, but the equation does not scatter to the transport problem.
    Considering the nonlinear problem as a forced linear problem and in view of
    the results of Deng-Masmoudi, \cite{dengmasmoudi2018}, we consider this as
    an important first step suggesting that nonlinear inviscid damping may
    similarly hold in lower regularity in spite of instability.
  \item Instead of scattering to the transport problem, we observe stability in
    a sum space consisting of \emph{transport-like} and \emph{stationary-like}
    behavior.
    However, in the consistency problem it is shown that for the nonlinear
    problem still further refinements are necessary.
  \item Stationary streamlines interacting with the forcing and ``shear
    behavior'' with respect to frequency in time (c.f Section \ref{sec:linEuler}) are interesting mathematical
    effects, which might also be of interest from a physical perspective.
  \item Even in the case of particularly simple forcing such a time-independent
    forcing, the consistency equation exhibits resonant behavior and instability
    in $H^s, s>-1$. However, the evolution is weakly asymptotically stable in $H^{-1}$.
  \item A key challenge in the consistency problem is given by the rather low
    regularity of (forced) solutions and the very different time-dependence of
    factors in the nonlinearity $v \cdot \nabla \omega$.
    In view of the question of persistence of inviscid damping, we hence focus
    on the analysis in negative Sobolev spaces.
  \end{itemize}
\end{rem}

Concerning the full nonlinear problem, we obtain a similar result for the
Navier-Stokes equations, where we further restrict the choice of (small) forcing.
That is, we choose $\int f(t,x,y) dx$ to fix the $x$ average of the vorticity
and for simplicity consider the case where $f- \int f dx$ is stationary.
\begin{thm}[The forced Navier-Stokes problem near Couette flow]
  Consider the linearized forced Navier-Stokes equations near Couette flow and
  let $\omega_0, f_0 \in H^s, s\geq 0$ with $\int \omega_0 dx =0 = \int f_0 dx$.
  Then in the case of a resonant forcing $f(t,x,y)=f_0(x+ty,y)$, there exists a
  decomposition
  \begin{align*}
    \omega(t,x,y)=\omega_1(t,x+ty,y)+ \omega_2(t,x+ty,y)
  \end{align*}
  such that both $\omega_1, \omega_2$ are stable in $L^2$ and $\omega_1$ exhibits enhanced
  dissipation, but $\omega_2$ generally only exhibits decay at an algebraic rate
  \begin{align*}
    \exp(C \nu t^3) \|\omega_1(t)\|_{L^2} + \|\omega_2(t)\|_{L^2} \leq C_\nu  (\|\omega_0\|_{L^2} + \|f_0\|_{L^2}).
  \end{align*}
  If the forcing is stationary, then there exists a stationary solution $g \in H^{1}$
  of the linearized problem and any solution with initial data $\omega_0$ is
  damped towards $g$ at super-exponential rates
  \begin{align*}
    \|\omega(t)-g\|_{L^2}\leq \exp(-\frac{\nu}{3}t)\|\omega_0-g\|_{L^2}.
  \end{align*}

  For the nonlinear forced Navier-Stokes equations around Couette flow, we
  consider $f \in L^{2}$ time-independent with $\|f\|_{L^{2}}\leq
  \frac{\nu}{40}$. Then there exists a stationary solution $g \in H^{1}$ of the equation
  \begin{align*}
    y \p_{x} g +\nu \Delta g + (v_{g}\cdot \nabla g)_{\neq} = f,
  \end{align*}
  where $()_{\neq}$ denotes the $L^2$ projection on functions with vanishing $x$-average.
  
  Furthermore, solutions $\omega$ of the forced Navier-Stokes equation with different
  initial data
  \begin{align*}
    \dt \omega + y \p_{x} \omega + \left(\nabla^{\bot}\phi \cdot \nabla \omega  \right)_{\neq} -\nu \Delta \omega =f,
  \end{align*}
  with $f$ chosen to ensure the vanishing $x$ average, can be decomposed as
  \begin{align*}
    \omega=\omega^{\star}+g,
  \end{align*}
  where $\omega^{\star}$ decays (at least) as $\exp(-\frac{\nu}{2}t)$ in
  $L^{2}$.
\end{thm}

\subsection{Organization of the Article}
\label{sec:outline}

The remainder of the article is organized as follows:
\begin{itemize}
\item In Section \ref{sec:couette} we fix notational conventions and
  discuss the prototypical case of Couette flow, $U(y)=y$, where explicit
  solution formulas are available. In particular, we see that the conditions of
  the theorem and lemma are necessary in that case.
\item In Section \ref{sec:linEuler} we study the problem for more general shear
  flows exhibiting linear inviscid damping and comment on extensions of these
  results, resonances and conditions on $f$. 
\item In Section \ref{sec:simple} we discuss the nonlinear viscous problem.
  Here, we for simplicity fix $\int \omega dx$ and consider stationary forcing.
\end{itemize}

\subsection*{Acknowledgments}
The author would like to thank the MPI MIS, where part of the project was written, for its hospitality.

\section{Model Case: Couette Flow}
\label{sec:couette}

As an instructive model, let us consider the linearization around Couette flow
$(U(y)=0, F=0)$ on $\T \times I$, for which the linearized equation greatly simplifies:
\begin{align}
  \label{eq:2}
  \begin{split}
    \dt \omega + y \p_x \omega &= f,\\
    \omega|_{t=0}&=\omega_0.
  \end{split}
\end{align}
As a further simplification, for the remainder of this section, we will assume
that the interval $I$ is chosen such that $|y|\geq 1$ on $I$.
In order to introduce ideas, we here consider two very specific cases of
forcing. As an at first sight somewhat artificial case, we consider $f_0 \in H^1$ with non-trivial dependence on $x$ and
\begin{align}
  \label{eq:3}
  f(t,x,y)=f_0(x-ty,y).
\end{align}
We call this the \emph{resonant} case, since here the evolution of
$W(t,x,y):=\omega(x+ty,y)$ involves
\begin{align*}
  \int_0^t f_0(x-(\tau- \tau)y,y) d\tau= t f_0(x,y)
\end{align*}
and thus mixing and oscillations are canceled due the resonance of the structure
of $f(t,x,y)$ with the underlying shear behavior. 
\\

As a second case, we consider $f$ to be \emph{time-independent}, that is
\begin{align}
  \label{eq:4}
  f(t,x,y)=f_0(x,y),
\end{align}
with $f_0$ as above.
Both cases should be considered as prototypical and to show that a more
fine-grained control than just of the $L^2$ norm is necessary to control the
long-term behavior of \eqref{eq:2}

\begin{rem}
  We remark that the stationary and resonant cases discussed in the
  introduction, Section \ref{sec:introduction}, can be reduced to these cases.
  For instance, if $f$ is a given time-periodic forcing, we may apply a Fourier
  transform in time to decompose it into $e^{ikct+ikx}\mathcal{F}_{x,t}f(c,k,y)$. Then
  $\omega_{k}(t,x,y):=e^{-ikct}\mathcal{F}_{x}\omega(t,k,y) e^{ikx}$ solves
  \begin{align*}
    \dt \omega_{k} + (y+\frac{c}{k}) \p_{x} \omega_{k}= \mathcal{F}_{x,t}f(c,k,y)e^{ikx}.
  \end{align*}
  Since the equation decouples in $k$, we may consider it a fixed parameter and
  interpret this equation as a Galilean transformation $y\mapsto y+\frac{c}{k}$
  of the time-independent setting \eqref{eq:4}.
  \\

  We in particular remark that in this setting instead of stationary streamlines
  $y=0$, streamlines of interest are those which move at a speed
  $y=-\frac{c}{k}$ matching the frequency in time.
\end{rem}

In the present special case of Couette flow, we can compute explicit solutions
and thus show that in both cases stability and scattering to the transport
problem fail in $H^s, s>0$. Moreover, in the first case, stability even fails in
$H^s, s>-1$.
However, (weak) linear inviscid damping, that is (weak) asymptotic stability in $H^{-1}$ holds.
Furthermore, we see that a variant of the resonant case naturally arises in the
study of the consistency problem.

\begin{prop}
  \label{prop:Couettecases}
  Let $\omega_0 \in L^2$ and $f_0 \in L^2$ with $\int f_0 dx =0$.
  Then the explicit solution of \eqref{eq:2} in the resonant case \eqref{eq:3}
  is given by
  \begin{align}
    \omega(t,x,y)=\omega_0(x-ty,y) +t f_0(x-ty,y).
  \end{align}
  In particular, unless $f_0$ has trivial dependence on $x$, it holds that
  \begin{itemize}
  \item The evolution is algebraically unstable in $H^s$ for any $s>-1$.
  \item The evolution is stable in $H^{-1}$ and weakly compact in that space. We
    interpret this as a weak analogue of linear inviscid damping.
  \end{itemize}
 
  If we instead consider the time-independent case \eqref{eq:4} and suppose
  that additionally
  \begin{align}
    \label{eq:5}
    g(x,y):=\frac{1}{y}\p_x^{-1} f_0(x,y) \in L^{2},
  \end{align}
  then the explicit solution of \eqref{eq:2} is given by
  \begin{align*}
    \omega(t,x,y)=\omega_0(x-ty,y) -g(x-ty,y)+g(x,y),
  \end{align*}
  and, unless $g$ has trivial dependence on $x$, it holds that:
  \begin{itemize}
  \item The evolution is algebraically unstable in $H^s$ for any $s>0$.
  \item While stability holds in $L^2$, asymptotic stability or scattering fail.
  \item The evolution is asymptotically stable in $H^s$ for any $s<0$. In
    particular, the associated velocity field strongly converges in $L^2$ as
    $t\rightarrow \infty$.
  \end{itemize}
\end{prop}

\begin{proof}
  Using the method of characteristics and Duhamel's formula, we obtain that the
  explicit solution of \eqref{eq:2} is given by
  \begin{align}
    \label{eq:6}
    \omega(t,x,y)=\omega_0(x-ty,y) + \int_0^t f(\tau,x-(t-\tau)y,y) d\tau.
  \end{align}
  In the resonant case, the integral simplifies to
  \begin{align}
    \label{eq:7}
    \int_0^t f_0(x-ty,y) d\tau = t f_{0}(x-ty,y).
  \end{align}
  The algebraic instability in $H^s,s>0$ follows immediately from this explicit
  formula. Concerning the behavior in $H^s,s<0$, we note that
  \begin{align*}
    t f_0(x-ty,y)= -\frac{d}{dy} g(x-ty,y) + (\p_yg)(x-ty,y).
  \end{align*}
  Using the characterization of $H^{-1}$ as the dual space of $H^1_0$ and
  integrating by parts, we hence obtain stability in $H^{-1}$ with a bound by $\|g\|_{L^2}+ \|\p_y
  g\|_{L^2}\leq \|g\|_{H^1}$. Furthermore, $g(x-ty,y)$ and $(\p_yg)(x-ty,y)$
  weakly converge to $0$ in $L^2$. Hence, weak compactness in $H^{-1}$ follows
  and $(\p_yg)(x-ty,y)$ even converges strongly in $H^{-1}$. Similarly, the
  algebraic instability of $g(x-ty,y)$ in $H^\sigma, s>0$ yields instability of
  $\omega$ in $H^s, s>-1$.
  \\

  In the time-independent case, the explicit solution is given by
  \begin{align*}
    \omega(t,x,y)=\omega_0(x-ty,y) -g(x-ty,y)+g(x,y)
  \end{align*}
  by the fundamental theorem of calculus. We remark that here the vanishing
  average of $f_0$ in $x$ is necessary to ensure that the integral
  \begin{align*}
    \int_0^t f_0(x-(t-\tau)y,y) d\tau = \int_{x}^{x-ty}\frac{1}{y} f(\xi, y) d\xi
  \end{align*}
  is a well-defined periodic function in $x$. We observe that the change of variables
  $(x,y)\mapsto (x-ty,y)$ is an isometry in $L^{2}$ and hence $\omega$ is stable
  in $L^{2}$ but not asymptotically stable. Furthermore, the evolution is
  algebraically unstable in $H^s,s>0$ and asymptotically stable in $H^s, s<0$.
\end{proof}

We thus see that algebraic instability in $L^2$ and asymptotic stability in
$H^{-1}$ can be compatible. One might object that the choice of ``resonant''
force $f$ is quite artificial. In particular, we note that ``non-resonant''
time-independent (or time-periodic) force fields result in better stability
behavior in Sobolev regularity.

However, such resonant forcing naturally appears in the consistency equation.
More precisely, we have seen that in the case of time-independent forcing,
$\omega$ is of the form
\begin{align*}
  \omega(t,x,y)= \omega_1(x-ty,y) + \omega_2(x,y).
\end{align*}
In the consistency problem we insert this decomposition of $\omega$ into the
nonlinearity $\nabla^{\bot}\phi \cdot \nabla \omega$ and obtain four
corresponding products, some of which will be of resonant type.
\begin{prop}
  \label{prop:couetteconsistency}
  Let $\omega(t,x,y)=\omega_{1}(x-ty,y)+ \omega_{2}(x,y)$, where
  $\omega_1,\omega_2, \p_x \omega_1, \p_x \omega_2 \in H^{3/2} \cap W^{1,\infty}$. We consider the
  Duhamel integral
  \begin{align}
    \label{eq:8}
    \int_{0}^{t} (\nabla^{\bot}\phi \cdot \nabla \omega)(\tau,x+(t-\tau)y, y) d\tau,
  \end{align}
  for the case of a finite periodic channel, $\T \times I$.
  \begin{itemize}
  \item If $\omega_2=0$, the Duhamel integral \eqref{eq:8} is bounded and convergent in $H^s, s<-\frac{1}{2}$.
  \item If $\omega_1=0$, the Duhamel integral is bounded and convergent in $H^s,s<0$.
  \item If $\omega_1, \omega_2 $ are non-trivial, the Duhamel integral is
    bounded in $H^{-1}$ and asymptotically converges weakly. However, it grows
    unbounded in any $H^s$ with $s>-1$.
  \end{itemize}
\end{prop}

\begin{proof}
  We consider the problem in Lagrangian variables and thus
  $W(t,x,y)=\omega(t,x+ty,y)$ can be split as
  \begin{align*}
    W(t,x,y)=\omega_1(x,y)+\omega_2(x+ty,y).
  \end{align*}
  We stress that in the present setting we could restrict to $\omega_1,\omega_2
  \in H^{s}$ for $s$ large in order to allow for easier proofs.
  However, in the general setting of Section \ref{sec:linEuler}, $\omega_1$ and
  $\omega_2$ will turn out to be time-dependent with suitable control only in
  $H^s, s<\frac{3}{2}$  and a bound by $\log(t)$ in $W^{1,\infty}$.
  Our main challenge in the following is thus to handle this lack of higher regularity.

  Expressing the Duhamel integral in terms of $W$, we have to estimate 
  \begin{align}
    \label{eq:25}
    \int_0^T \nabla^\bot \Phi(t, x,y) \cdot \nabla W(t,x,y) d\tau, \\
 \notag \Delta_t \Phi:=  (\p_x^2+(\p_y-t\p_x)^2)\Phi(t,x,y)=W(t,x,y).
  \end{align}
  Splitting the stream function as
  \begin{align*}
    \Phi_1(t,x,y) &= \Delta_t^{-1} \omega_1(x,y), \\
    \Phi_2(t,x,y) &= \Delta_t^{-1}\omega_2(x+ty,y) = (\Delta^{-1}\omega_2)(x+ty,y),
  \end{align*}
  we split \eqref{eq:25} into four terms corresponding to different products.
  \begin{align*}
    & \quad \int_0^T \nabla^\bot \Phi(t, x,y) \cdot \nabla W(t,x,y) d\tau \\
    &=  \int_0^T \nabla^\bot\Delta_t^{-1} \omega_1(x,y)  \cdot \nabla \omega_1(x,y)  d\tau \\
    & \quad +  \int_0^T \nabla^\bot (\Delta^{-1}\omega_2)(x+ty,y) \cdot \nabla \omega_2(x+ty,y)  d\tau \\
    & \quad + \int_0^T \nabla^\bot\Delta_t^{-1} \omega_1(x,y)  \cdot \nabla \omega_2(x+ty,y)  d\tau \\
    & \quad +  \int_0^T \nabla^\bot (\Delta^{-1}\omega_2)(x+ty,y) \cdot \nabla \omega_1(x,y)  d\tau \\
    & = I + II +III + IV.
  \end{align*}
  We remark that if $\omega_1=0$, all terms except $II$ vanish. If $\omega_2=0$,
  all terms accept $I$ vanish. If both are non-trivial additionally $III$ and
  $IV$ are non-trivial. Here, we show $III$ to be bounded in $L^2$, while $IV$
  exhibits growth in $H^s$ for any $s>-1$.
\\

\underline{Ad I}:  We consider 
\begin{align}
  \label{eq:26}
    \int_0^T \nabla^\bot \Delta_t^{-1}\omega_1(x,y) \cdot \nabla \omega_1(x,y) dt.
  \end{align}
  We note that $\nabla^\bot \Delta_t^{-1}\omega_1(x,y)$ asymptotically decays in
  time, but that our bound
  \begin{align*}
    \|\nabla^\bot \Delta_t^{-1}\omega_1(x,y)\|_{L^2} \leq C (1+t)^{-\sigma} \|\omega_1\|_{H^{1+\sigma}}
  \end{align*}
  is not sufficient to obtain an integrable decay rate, since we are limited to
  $1+\sigma < \frac{3}{2}$.

  One way to avoid this obstacle is to consider the asymptotic behavior in
  negative Sobolev regularity. That is, expressing the nonlinearity as a
  divergence, we can interpolate:
    \begin{align*}
    \|\nabla^\bot \cdot  (\Delta_t^{-1}\omega_1(x,y) \cdot \nabla \omega_1(x,y))\|_{H^{-1}_0}&\leq \| \Delta_t^{-1}\omega_1(x,y)\|_{L^2} \|\nabla \omega_1(x,y)\|_{L^\infty}\\
    &\leq C t^{-3/2} \|\omega_{1}\|_{H^{3/2}} \|\nabla \omega_1(x,y)\|_{L^\infty} \\
    \|\nabla^\bot (\Delta_t^{-1}\omega_1(x,y) \cdot \nabla \omega_1(x,y))\|_{L^2} &\leq   C t^{-1/2} \|\omega_{1}\|_{H^{3/2}} \|\nabla \omega_1(x,y)\|_{L^\infty}
    \end{align*}
    to obtain that for $s\in [-1,0]$
    \begin{align*}
      \|\nabla^\bot \cdot  (\Delta_t^{-1}\omega_1(x,y) \cdot \nabla \omega_1(x,y))\|_{H^{s}_0} \leq C t^{-1/2+s} \|\omega_{1}\|_{H^{3/2}} \|\nabla \omega_1(x,y)\|_{L^\infty},
    \end{align*}
    and is hence integrable in time for $s< -\frac{1}{2}$. However, this
    estimate is not optimal.

  As can be seen via characterization in terms of Fourier variables for the
  setting of an infinite channel or by a basis characterization for a finite
  channel (c.f. \cite{Zill5}), the linear operator
  \begin{align*}
  B_T:=  \int_0^T \Delta_t^{-1} dt
  \end{align*}
  is uniformly bounded in $L^2$ and converges as $T \rightarrow \infty$.
  Hence, we may compute \eqref{eq:26} explicitly as
  \begin{align*}
    \nabla \omega_1(x,y) B_T \nabla^\bot\omega_1(x,y).
  \end{align*}

  \underline{Ad II:} We explicitly compute $II$ as 
  \begin{align*}
    & \quad \int_0^T \nabla_t^\bot \Delta_t^{-1} \omega_2(x+ty,y)  \cdot \nabla_t \omega_2(x+ty,y) dt \\
    &= \int_0^T (\nabla^\bot \Delta^{-1} \omega_2 \cdot \nabla \omega_2) (x+ty,y) dt \\
    &= \frac{1}{y}\p_x^{-1} (\nabla^\bot \Delta^{-1} \omega_2 \cdot \nabla \omega_2) (x+ty,y)\bigg|_{t=0}^T.
  \end{align*}
  We in particular note that $II$ consists of one term depending on $(x+Ty,y)$,
  which is hence stable in Lagrangian Sobolev spaces but not Eulerian, and another
  term depending on $(x,y)$, where the converse holds.
  Thus, we emphasize that stability should more naturally be considered in sum
  spaces.
  This concludes the proof in the case where one of $\omega_1$ or $\omega_2$ is
  zero.
  When considering the general case, one additionally has to estimate the
  \emph{mixed} terms $III$ and $IV$.
  \\

  \underline{Ad III:} We express the integral as 
  \begin{align*}
    \int_0^T \nabla_t^\bot \Delta_t^{-1} \omega_1(x,y) \cdot \nabla_t \omega_2(x+ty,y) dx \\
    = \int_0^T \nabla_t^\bot \Delta_t^{-1} \omega_1(x,y) \cdot (\nabla \omega_2) (x+ty,y)
  \end{align*}
  Note that, even in the whole space case, the multiplier associated with
  $\nabla_t \Delta_t^{-1}$ is not integrable in time, but rather suggests a
  logarithmic growth.
  However, we may additionally make use of the oscillation of the second factor
  and compute:
  \begin{align*}
    \nabla_t^\bot \Delta_t^{-1} \omega_1(x,y) \cdot \frac{1}{y}\p_x^{-1}(\nabla \omega_2) (x+ty,y) \bigg|_{t=0}^T \\
    - \int_0^T \p_t(\nabla_t^\bot \Delta_t^{-1}) \omega_1(x,y) \cdot \frac{1}{y}\p_x^{-1}(\nabla \omega_2) (x+ty,y).
  \end{align*}
  We note that $\p_t(\nabla_t^\bot \Delta_t^{-1}) \approx \p_x \Delta_t^{-1}$
  exhibits higher decay rates in $L^2$ when applied to sufficiently regular
  functions. In particular,
  \begin{align*}
    \|\p_t(\nabla_t^\bot \Delta_t^{-1}) \omega_1(x,y)\|_{L^2} \leq C (1+t)^{-3/2+\delta} \|\omega_1\|_{H^{3/2-\delta}}
  \end{align*}
  and hence the latter integral is absolutely convergent in $L^2$.
  \\

  \underline{Ad IV:} We split this contribution as
  \begin{align*}
    \int_0^T \nabla^\bot \Delta_t^{-1} \omega_2(x+ty,y) \nabla \omega_1(x,y) dt\\
    = \int_0^T \nabla^\bot \phi_2(x+ty,y) \cdot \nabla \omega_1(x,y) dt \\
    = \int_0^T (\nabla^\bot \phi_2)(x+ty,y) \cdot \nabla \omega_1(x,y) dt \\
    + \int_0^T t \p_x \phi_2(x+ty,y) \p_x \omega_1(x,y) dt. 
  \end{align*}
  The first integral has an explicit anti-derivative with
  \begin{align}
    \label{eq:27}
    \frac{1}{y}\p_x^{-1}(\nabla^\bot \phi_2)(x+ty,y)\bigg|_{t=0}^T \cdot \nabla \omega_1(x,y). 
  \end{align}
  Similarly, for the second integral we obtain
  \begin{align}
    \label{eq:28}
    \frac{T}{y} \phi_2(x+Ty,y) \p_x \omega_1(x,y) \\
   \notag  - \int_0^T \frac{1}{y} \phi_2(x+ty,y)\p_x \omega_1(x,y) dx.
  \end{align}
  The last integral is then again of the above form and can be computed as
  \begin{align}
    \label{eq:29}
     \frac{1}{y}\p_x^{-1} \phi_2(x+ty,y)\bigg|_{t=0}^T \p_x \omega_1(x,y).
  \end{align}
  We remark that the contributions in \eqref{eq:27} and \eqref{eq:29} are of the
  form
  \begin{align*}
    a_1(x+ty,y) a_2(x,y) + b_1(x,y) b_2(x,y).  
  \end{align*}
  The latter term is time-independent and hence stable, while the first term is
  of a \emph{mixed} or \emph{product} type. In particular, this term is not
  asymptotically stable in $L^2$, neither with respect to Eulerian or
  Lagrangian coordinates, though each factor is asymptotically stable with
  respect to one of the coordinates.
  However, by the Riemann-Lebesgue Lemma the product converges weakly in $L^2$
  as $T \rightarrow \infty$ and strongly in $H^s,s <0$.

  The contribution in \eqref{eq:28}, is seen to be in general algebraically
  unstable in $H^s$, $s>-1$, similarly to the resonant case of Proposition
  \ref{prop:Couettecases}.
  However, we may again note that
  \begin{align*}
    \frac{T}{y} \phi_2(x+Ty,y) = \frac{1}{y^2} (\frac{d}{dy}-\p_y) \p_x^{-1} \phi_2(x+Ty,y),
  \end{align*}
  which can be used to establish uniform boundedness and weak compactness in
  $H^{-1}_0$.
  We remark that, at this point we require that $\p_y \p_x
  \omega_1(x,y) \in L^\infty$, so that $\p_x \omega_1(x,y) \psi \in H^{1}_0$,
  whenever $\psi \in H^{1}_0$.
   
\end{proof}

\section{The Forced Linear Euler Equations}
\label{sec:linEuler}

Following this preparatory example, let us now consider the case of a more
general shear flow.
\begin{align}
  \label{eq:20}
  \begin{split}
    \dt \omega + U(y)\p_{x} \omega - U''(y)\p_{x}\phi&= f, \\
    \Delta \phi &= \omega,
  \end{split}
\end{align}

We denote the solution operator of the unforced linearized Euler equations in
Lagrangian coordinates by $S(t_1,t_2)$.
\begin{defi}
  Let $U(y)\in C^{3}$ be given and define $S(t_{2},t_{1}):L^{2}\rightarrow
  L^{2}$ to be the solution operator of the \emph{unforced scattering problem}.
  That is, consider the case $f=0$ and define $W(t,x,y):=\omega(t,x+tU(y),y)$.
  Then $W$ satisfies equation
  \begin{align}
    \label{eq:18}
    \dt W - U'' \p_{x} (\p_{x}^{2}+(\p_{y}-tU'(y)\p_{x})^{2})^{-1} W =0.
  \end{align}
  We define the associated \emph{solution operator} $S$ by
  $S(t_{2},t_{1})W(t_{1})=W(t_{2})$.
\end{defi}
We stress that establishing the stability of $S(t_{2},t_{1})$ in Sobolev spaces
$H^{s}$ is one of the main results and challenges of works on linear inviscid
damping (c.f. \cite{CZZ17}, \cite{Zill3}, \cite{Zill6}, \cite{wei2017linear},
\cite{wei2018linear}). In this work we build up on these results to study the forced equations.
In particular, we make use of the following properties of $S(t_1,t_2)$, which are for
instance established for a family bilipschitz flows in \cite{Zill5} and for more
general flows in \cite{Zhang2015inviscid}:
\begin{align}
  \label{eq:9}
  \|S(t_{2},t_{1})\|_{H^{s}\rightarrow H^{s}} &\leq C_{s} \text{ for } 0\leq t_{1}\leq t_{2} < \infty, s \in (-\frac{3}{2},\frac{3}{2}), \\
  \label{eq:10}
  \|(\p_{t}S(t,0)) u \|_{H^{s}} &\leq C_{s} (1+|t|)^{-\sigma}\|S(t,0) u \|_{H^{s+\sigma}} \text{ for }s \in (-3,3), \sigma \in (0,2],\\
  \label{eq:11}
  S(t_{2},t_{1}) f(x+\tau U(y),y) &= (S(t_{2}-\tau,t_{1}-\tau)f(x,y)) (x+\tau U(y),y).
\end{align}
We remark that \eqref{eq:9} in the above cited works is proven for $0\leq s <
\frac{3}{2}$ with the upper bound being sharp due to boundary effects.
The estimates for the spaces $H^s_0,s<0$ follow by duality.
Furthermore, we for simplicity assume for the remainder of this section that $I$ is chosen in such a way that
\begin{align}
U(y)\text{ is bounded above and below on }I,
\end{align}
in order to avoid discussion of stability in degenerately weighted spaces (c.f. \cite{CZZ17}).
We note that \eqref{eq:11} is a consequence of the structure of the equation, since
all coefficient functions depend on time by conjugation with the transport
dynamics. The other two properties quantify stability of the unforced problem in
Sobolev regularity as well as the damping rates of the associated velocity
field.
\\
  
We obtain the following integral formula for the solution to the forced
problem by applying Duhamel's principle.
\begin{lem}
  The solution of the forced scattering equation
  \begin{align}
    \label{eq:12}
    \begin{split}
      \dt W - U'' \p_{x} (\p_{x}^{2}+(\p_{y}-tU'(y)\p_{x})^{2})^{-1} W &= f(t,x-tU(y),y), \\
      W|_{t=0}&=\omega_{0}
    \end{split}
  \end{align}
  is given by
  \begin{align}
    \label{eq:13}
    \begin{split}
    W(t)& = S(t,0)\omega_{0} + \int_{0}^{t} S(t,\tau) f(\tau, x-\tau U(y),y) d\tau \\
    &=S(t,0)\omega_{0} + \int_{0}^{t} (S(t-\tau,0) f(\tau,\cdot,\cdot))(x-\tau y, y) d\tau
    \end{split}
  \end{align}
\end{lem}
We remark that $f(t,x,y)$ denotes the forcing in Eulerian variables and
\eqref{eq:12} is stated in Lagrangian variables. 

\begin{proof}
  Let $L_t$ denote the operator
  \begin{align*}
  L_t u= - U'' \p_{x} (\p_{x}^{2}+(\p_{y}-tU'(y)\p_{x})^{2})^{-1}u.
  \end{align*}
  Then it holds that
  \begin{align*}
    (\p_t- L_t)S(t,\tau) u =0
  \end{align*}
  for any $t,\tau \in \R$ and $u \in L^2$.
  Hence, it follows by direct computation that
  \begin{align*}
    & \quad (\p_t- L_t) \left( S(t,0)\omega_{0} + \int_{0}^{t} S(t,\tau) f(\tau, x-\tau U(y),y) d\tau  \right) \\
    &= 0 + \int_0^t 0 d\tau + S(0,0)f(t,x-t U(y),y) = f(t,x-t U(y),y).
  \end{align*}
  Concerning the interaction with the transport problem, we note that the
  time-dependence of all coefficient functions and the elliptic operator is given
  by conjugation with the transport operator, which is also formulated in
  \eqref{eq:11}. Hence, this conjugation is equivalent to a time-shift of the
  solution operator $S(\cdot, \cdot)$.
\end{proof}

We note that the stability and asymptotic behavior of $S(t,0)\omega_{0}$ has
been intensively studied in prior works on the unforced problem. In the
following we hence restrict to the case $\omega_{0}=0$ and studying the Duhamel
integral. Furthermore, we note that the problem decouples with respect to
Fourier modes $k$ in $x$ and thus consider
\begin{align}
  \begin{split}
    & \quad \int_{0}^{t} S(t,\tau) e^{iktU(y)}\hat{f}(\tau,k,y)e^{ikx} d\tau \\
    &= \int_{0}^{t} e^{ik\tau U(y)} S(t-\tau,0) \hat{f}(\tau,k,y) e^{ikx} d\tau,
  \end{split}
\end{align}
where $k \in \Z \setminus \{0\}$ is arbitrary but fixed.

Following a similar strategy as in Section \ref{sec:couette}, we are interested
in the behavior of \eqref{eq:13} for the following model cases of forcing.
\begin{defi}[Model cases]
  \label{defi:modelcases}
  Let $S(t,\tau)$ denote the solution operator of the unforced problem and let
  $f(t,x,y)$ be a given forcing. We then introduce the following model cases:
\begin{itemize}
\item We call the forcing $f$ \emph{resonant}, if there exists $f_0$
  such that $f(t,x,y)= (S(t,0)f_0)(x+ty,y)$. 
\item We call the forcing $f$ \emph{stationary}, if there exists $f_0$
  such that $f(t,x,y)=(S(0,-t)f_0)(x,y)$. 
\end{itemize}
\end{defi}

Similarly as in Section \ref{sec:couette}, we observe that the at first sight
more general settings of $f(t,x,y)$ or $f(t,x+tU(y),y)$ being periodic can be
reduced to these cases. Following the same Fourier decomposition argument we may
restrict our discussion to forces $e^{ict}f_{c}(x,y)$ or
$e^{ict}f_{0}(x+tU(y),y)$ and after applying a Galilean transform we may further
reduce to $c=0$.
Now note that
\begin{align*}
  f(x,y)= S(t,0) (S(0,t)f(x,y))
\end{align*}
and that $f_{0}(t,x,y):=(S(0,t)f(x,y))$ enjoys the same $H^{s}$ regularity as $f$ by
\eqref{eq:9} and that $\p_{t} f_{0}$ decays in $L^{2}$ with algebraic
rates by property \eqref{eq:10}.
Hence, these cases are very slight generalizations of Definition
\ref{defi:modelcases} with $f_{0}$ depending on time, as discussed in  Proposition \ref{prop:conistency}.
The following proposition considers the case with $f_{0}$ independent of time to
allows for a more transparent characterization via Duhamel's formula.

\begin{prop}
  \label{prop:Duhamel}
  Let $\omega_0 \in L^2$ and suppose that $f$ is in the resonant case
  with $f_0 \in H^1$.
  Then the explicit solution of \eqref{eq:20} is given by
  \begin{align}
    \omega(t,x,y)=(S(t,0)\omega_0)(x-ty,y) +t (S(t,0) f_0)(x-ty,y).
  \end{align}
  In particular, unless $f_0$ has trivial dependence on $x$, it holds that
  \begin{itemize}
  \item The evolution is algebraically unstable in $H^s$ for any $s>-1$.
  \item The evolution is stable in $H^{-1}$ and weakly compact in that space. We
    interpret this as a weak analogue of linear inviscid damping.
  \end{itemize}
 
  If we instead consider the case where $f$ is in the stationary setting and $\int f_0
  dx =0$, then the
  explicit solution of \eqref{eq:20} is given by
  \begin{align}
    \label{eq:22}
    \omega(t,x,y)=(S(t,0)(\omega_0-g))(x-ty,y) +S(0,-t)g(x,y),
  \end{align}
  where $g$ solves
  \begin{align*}
    U(y)\p_x g (x,y)=f_0(x,y).
  \end{align*}
  Hence, it holds that
  \begin{itemize}
  \item The evolution is algebraically unstable in $H^s$ for any $s>0$.
  \item While stability holds in $L^2$, asymptotic stability or scattering fail.
  \item The evolution is asymptotically stable in $H^s$ for any $s<0$. In
    particular, the associated velocity field strongly converges in $L^2$ as
    $t\rightarrow \infty$.
  \end{itemize}
\end{prop}

\begin{proof}
  Recall that the explicit solution in Lagrangian variables is given by \eqref{eq:13}:
  \begin{align*}
     W(t)& = S(t,0)\omega_{0} + \int_{0}^{t} S(t,\tau) f(\tau, x-\tau U(y),y) d\tau.
  \end{align*}
  In the resonant case we hence obtain that
  \begin{align*}
    W(t)& = S(t,0)\omega_{0} + \int_{0}^{t} S(t,\tau) S(\tau,0) f_0(x,y) d\tau \\
        &= S(t,0)\omega_{0} + \int_{0}^{t} S(t,0)f_0(x,y) d\tau \\
        &=  S(t,0)\omega_{0} + t  S(t,0)f_0(x,y).
  \end{align*}
  Similarly, in the stationary case we obtain that
  \begin{align*}
    W(t)& = S(t,0)\omega_{0} + \int_{0}^{t} S(t,\tau) (S(0,-\tau) f_0) ( x-\tau U(y),y) d\tau \\
        & = S(t,0)\omega_{0} + \int_{0}^{t} S(t,\tau) (S(\tau,0) f_0( x-\tau U(y),y)) d\tau \\
        &=  S(t,0)\omega_{0} + S(t,0) \int_{0}^{t} f_0( x-\tau U(y),y) d\tau \\
        &= S(t,0)(\omega_{0}-g) + S(t,0) g(x-tU(y),y)
  \end{align*}
  The claimed solution formulas then follow by the change to Eulerian variables
  and noting that $S(t,0) g(x-tU(y),y) = (S(0,-t)g)(x-tU(y),y)$.
\end{proof}

\subsection{Consistency}
\label{sec:consistency}

When considering the consistency problem, similarly as in Section
\ref{sec:couette}, we encounter terms which we interpret as further forcing.
However, unlike in the Couette flow case, the solution operators $S(t,\tau)$ are
generally not compatible with products. Hence, these contributions are similar
to the resonant and stationary cases discussed above, but also include explicit
time-dependences.

\begin{prop}
  \label{prop:conistency}
  Let $\omega(t), \phi(t)$ denote the solution of the linearized forced Euler
  equations \eqref{eq:20}, where $f$ is in the stationary case.
  Further assume that $\omega_0, g, \p_x \omega_0, \p_x g \in
  H^{\frac{3}{2}}\cap W^{1,\infty}$.
  Then the Duhamel integral
  \begin{align*}
    \sigma(t,x+tU(y),y):= \int_0^t S(t,\tau)(\nabla^\bot \phi \cdot \nabla \omega)(\tau,x-\tau U(y), y) d\tau
  \end{align*}
  satisfies the following properties:
  \begin{enumerate}
  \item If either $g=0$ or $\omega_0=g$, then $\sigma$ is asymptotically stable
    in $H^s$ for any $s<-\frac{1}{2}$.
  \item Generally, $\sigma(t,x+tU(y),y)$ is uniformly bounded in $H^{-1}$ and
    weakly asymptotically convergent, but
    not in $H^s, s>-1$.
  \end{enumerate}
\end{prop}

\begin{proof}
  We recall that $\omega(t,x,y)$ is explicitly given by \eqref{eq:22}
  \begin{align*}
    \omega(t,x,y)=(S(t,0)(\omega_0-g))(x-tU(y),y) +S(0,-t)g(x,y),
  \end{align*}
  which, in analogy to Proposition \ref{prop:couetteconsistency}, we rewrite as 
  \begin{align*}
    \omega(t,x,y) &= \omega_1 (t, x-tU(y),y) + \omega_2(t,x,y) \\
    \Leftrightarrow W(t,x,y)&= \omega_1(t,x,y) + \omega_2(t,x+tU(y),y). 
  \end{align*}
  We recall that by properties \eqref{eq:9} and \eqref{eq:10}, the time
  derivatives of $\omega_{1}(t,x,y)$ and $\omega_{2}(t,x,y)$ decay in $L^{2}$
  with algebraic rates, depending on the regularity of $\omega_0-g$ and $g$.

  Following a similar strategy as in Section \ref{sec:couette}, we further
  split the stream function as 
  \begin{align*}
    \phi(t,x+tU(y),y)= \phi_1(t,x,y) + \phi_2(t,x+tU(y),y), 
  \end{align*}
  where $\phi_1,\phi_2$ are the solutions of
  \begin{align*}
    (\p_x^2+(\p_y-tU'(y)\p_x)^2)\phi_1&=\omega_1, \\
    \Delta \phi_2&= \omega_2.
  \end{align*}
  We note that $\phi_2$ asymptotically decays in $L^2$ by inviscid damping and
  does so with algebraic rates by \eqref{eq:9} and \eqref{eq:10}.
  In contrast $\phi_1$ is not expected to decay but rather behave similar to a
  stationary function.

  Using this splitting we decompose $\sigma(t)$ into four terms:
  \begin{align}
    \begin{split}
    \sigma(t,x+tU(y),y)&= \int_0^tS(t,\tau)\left(\nabla^\bot \phi_1(\tau, x,y) \cdot  \nabla \omega_1(\tau, x, y)\right) d\tau \\
                 & \quad + \int_0^tS(t,\tau)\left(\nabla^\bot \phi_2(\tau, x+\tau U(y),y) \cdot  \nabla\omega_2(\tau, x+\tau U(y), y)\right) d\tau \\
                 & \quad + \int_0^tS(t,\tau)\left(\nabla^\bot_\tau \phi_1(\tau, x, y) \cdot  \nabla_\tau \omega_2(\tau, x+\tau U(y), y)\right) d\tau \\
                 & \quad + \int_0^tS(t,\tau)\left(\nabla^\bot \phi_2(\tau, x+\tau U(y),y) \cdot  \nabla\omega_1 (\tau,x, y)\right) d\tau \\
                 &=: I + II + III + IV.
    \end{split}
  \end{align}
  It remains to show that these four terms behave similarly as in the Couette
  case of Proposition \ref{prop:Couettecases}.
  Compared to that setting, we here additionally have to control corrections due to
  $S(t,\tau)$ and the time-dependence of $\omega_1,\omega_2$ and account for the
  fact that the stability estimates \eqref{eq:9} are only valid for
  $H^s,s<\frac{3}{2}$ and that $\|\p_y \omega\|_{L^\infty} \leq C \log(2+t)$ for
  $\omega \in \{\omega_1,\p_x \omega_1,\omega_2,\p_x \omega_2\}$ (c.f.
  \cite{Zill5}, \cite{wei2018linear}).
  We claim that
  \begin{itemize}
  \item $I, III$ are bounded and convergent in $H^s, s<-\frac{1}{2}$. 
  \item $III$ is bounded in $L^2$ and convergent in $H^s,s<0$.
  \item $IV$ is uniformly bounded in $H^{-1}$ and weakly asymptotically convergent.
  \end{itemize}

  \underline{Ad I: } We remark that this case in a sense corresponds to the unforced
  problem. In \cite{Zill5} it has been discussed for the setting without
  boundary, where stability in higher Sobolev regularity is available, and in
  the context of a blow-up result for the setting with boundary.
  In the present setting, we instead aim to establish stability in negative
  Sobolev regularity $H^s, s<-\frac{1}{2}$, despite the blow-up behavior and
  when $\omega_1$ is of comparably low regularity.

  Since $S(\cdot,\cdot)$ is a bounded linear operator on $H^s$, it suffices to
  show that
  \begin{align}
    \label{eq:17}
    \int_0^t \| \nabla^\bot \phi_1(\tau) \cdot \nabla \omega_1(\tau) \|_{H^s} d\tau
  \end{align}
  is uniformly bounded in $t$.
  If higher regularity estimates were available, we could estimate by
  \begin{align*}
    \|\nabla^\bot \phi_1(\tau)\|_{L^2} &\leq (1+|\tau|)^{-\sigma} \|\omega_1(\tau)\|_{H^{1+\sigma}}, \\
    \|\nabla \omega_1(\tau)\|_{L^\infty} &\leq \|\omega_1(\tau)\|_{H^{\sigma'}},
  \end{align*}
  with $1<\sigma\leq 2$ and $\sigma'>\frac{3}{2}$. However, as shown for
  instance in \cite{Zill5}, both estimates fail in the setting with boundary in the sense
  that $\|\omega_1(\tau)\|_{H^\sigma}$ for $\sigma>\frac{3}{2}$ and
  $\|\p_y\omega_1(\tau)\|_{L^\infty}$ generally grow unbounded as $\tau
  \rightarrow \infty$.
  
  Using the null structure of the nonlinearity, we obtain the following estimate
  by interpolation for any $s \in [-1,0]$:
  \begin{align*}
    \|\nabla^\bot \phi_1(\tau) \cdot \nabla \omega_1(\tau)\|_{L^2} &\leq \|\nabla^\bot \phi_1(\tau)\|_{L^2} \|\nabla \omega_1(\tau)\|_{L^\infty}, \\
    \|\nabla^\bot \phi_1(\tau) \cdot \nabla \omega_1(\tau)\|_{H^{-1}} &= \| \nabla^{\bot} \cdot (\phi_1(\tau) \cdot \nabla \omega_1(\tau))\|_{H^{-1}} \leq \|\phi_1(\tau)\|_{L^2} \|\nabla \omega_1(\tau) \|_{L^\infty} \\
    \Rightarrow \|\nabla^\bot \phi_1(\tau) \cdot \nabla \omega_1(\tau)\|_{H^s} & \leq  \|\phi_1(\tau)\|_{H^{s+1}} \|\nabla \omega_1(\tau) \|_{L^\infty}.
  \end{align*}
  We further recall that, by \eqref{eq:10}, for $\sigma \in [0,2]$ 
  \begin{align*}
    \|\phi_1(\tau)\|_{H^{s+1}} \leq C(1+|\tau|)^{-\sigma} \|\omega_1(\tau)\|_{H^{s+1+\sigma}}
  \end{align*}
  and further use that by the more detailed analysis of the extremal case
  $\sigma=\frac{3}{2}$ (c.f. \cite{CZZ17}, \cite{Zill5}, \cite{wei2018linear})
  it holds that
  \begin{align*}
    \|\nabla \omega_1(\tau) \|_{L^\infty} \leq C \log(2+\tau).
  \end{align*}
  Hence, choosing $s<-\frac{1}{2}$ and $1< \sigma \leq 2$ such that
  $s+1+\sigma<\frac{3}{2}$ it follows that \eqref{eq:17} is uniformly bounded in
  $t$.
  \\
  
  \underline{Ad II:} We introduce the short notation 
  \begin{align*}
  f_2(t,x,y):= \nabla^\bot \phi_2(t,x,y) \cdot  \nabla \omega_2(t, x, y) .
  \end{align*}
  Then $I$ may be written as
  \begin{align*}
    \int_0^tS(t,0) S(0, \tau) f_2(\tau, x -\tau U(y),y) d\tau = S(t,0) \int_0^t (S(-\tau, 0) f_2)(\tau, x -\tau U(y),y) d\tau.
  \end{align*}
  We note that this term is very similar to forcing in the
  stationary case except that 
  \begin{align*}
    (S(-\tau, 0) f_2)(\tau,x,y) =: f_2^\star(\tau, x,y)
  \end{align*}
  explicitly depends on time.
  Using integration by parts and assuming zero average in $x$,
  we may compute the time-integral as
  \begin{align*}
    & \quad S(t,0)  \frac{1}{U(y)}\p_x^{-1}\left(f_2^\star(t,x-t U(y),y) - f_2^\star(0,x,y)\right) \\
    & - S(t,0) \int_0^t \frac{1}{U(y)}\p_x^{-1}\p_\tau f_2^\star (\tau,x-\tau U(y), y) d\tau. 
  \end{align*}
  The term in the first line is stable in a sum space and asymptotically
  converges in $H^{s}, s<0$.
  It thus remains to discuss the integral term, which involves the time
  dependence of $f^\star_2$ on $f_2$ via $S(-\tau,0)$ and of $f_2$.
  If $f_2$ were sufficiently regular, one could estimate this term using the
  algebra property of $H^\sigma, \sigma> 1$ and algebraic decay rates in
  \eqref{eq:10}.
  However, as $S(\cdot,\cdot)$ is only bounded in $H^s, |s|< \frac{3}{2}$ by
  \eqref{eq:9} we again instead consider negative Sobolev spaces $H^s,
  s<-\frac{1}{2}$.
  In that case, we may choose $1<\sigma\leq 2$ such that $s+\sigma
  <\frac{1}{2}$ and estimate
  \begin{align*}
    \|(\p_\tau S(-\tau,0))f_2\|_{H^s} \leq C (1+|\tau|)^{-\sigma} \|f_2\|_{H^{s+\sigma}}
  \end{align*}
  In particular, we may then use a Sobolev embedding and the null structure of the nonlinearity 
  \begin{align*}
    \nabla^\bot \phi_2 \cdot  \nabla \omega_2 = \nabla \cdot ( \nabla^\bot \phi_2  \omega_2 ) = \nabla^\bot \cdot (\phi_2 \nabla \omega_2),
  \end{align*}
  to control $\|f_2\|_{H^{s+\sigma}}$ in terms of the initial data.

  Concerning the time derivative of $f_2$, by the mapping properties of
  $S(\cdot,\cdot)$ it suffices to estimate $\|\dot f_2(\tau)\|_{H^s}$.
  Here, as in case $I$, we use the structure of the nonlinearity to estimate
  \begin{align*}
    \|\dot f_2\|_{H^s} \leq \|\dot \phi_2\|_{H^{s+1}} \|\nabla \omega_2\|_{L^\infty}  + \|\nabla \phi_2\|_{L^\infty} \|\dot \omega_2\|_{H^{s+1}}.
  \end{align*}
  The integrand is hence decaying with integrable rates in $H^{s}, s<
  -\frac{1}{2}$, and in particular convergent.
  \\
   
  \underline{Ad III:}
  We express the integral as
  \begin{align*}
    \int_0^t S(t,\tau) \nabla_\tau^\bot \Delta_\tau^{-1} \omega_1(\tau, x,y) \cdot (\nabla \omega_2)(\tau, x+\tau U(y),y) d\tau
  \end{align*}
  Integrating by parts in the $x+\tau U(y)$ dependence, we obtain
  \begin{align*}
    S(t,\tau) \nabla_\tau^\bot \Delta_\tau^{-1} \omega_1(\tau, x,y) \frac{1}{U(y)}\p_x^{-1}(\nabla \omega_2)(\tau, x+\tau U(y),y) \bigg|_{t=0}^T,
  \end{align*}
  which is asymptotically stable in $H^s, s<0$,  as well as several error terms
  \begin{align*}
    & \quad \int_0^t \dot S(t,\tau) \nabla_\tau^\bot \Delta_\tau^{-1} \omega_1(\tau, x,y) \cdot \frac{1}{U(y)}\p_x^{-1} (\nabla \omega_2)(\tau, x+\tau U(y),y) d\tau \\
    & + \int_0^t S(t,\tau) (\nabla_\tau^\bot \Delta_\tau^{-1})' \omega_1(\tau, x,y) \cdot \frac{1}{U(y)}\p_x^{-1} (\nabla \omega_2)(\tau, x+\tau U(y),y) d\tau \\
    & + \int_0^t S(t,\tau) \nabla_\tau^\bot \Delta_\tau^{-1} \dot \omega_1(\tau, x,y) \cdot \frac{1}{U(y)}\p_x^{-1} (\nabla \omega_2)(\tau, x+\tau U(y),y) d\tau \\
    & + \int_0^t S(t,\tau) \nabla_\tau^\bot \Delta_\tau^{-1} \omega_1(\tau, x,y) \cdot \frac{1}{U(y)}\p_x^{-1} (\nabla \dot \omega_2)(\tau, x+\tau U(y),y) d\tau.
  \end{align*}
  We remark that the second and third term can be estimated in $L^2$ using that
  \begin{align*}
    & \quad \|(\nabla_\tau^\bot \Delta_\tau^{-1})' \omega_1(\tau, x,y)\|_{L^2} + \| \nabla_\tau^\bot \Delta_\tau^{-1} \dot \omega_1(\tau, x,y)\|_{L^2} \\
    &\leq C (1+\tau)^{-3/2+\delta} (\|\omega_1\|_{H^{3/2-\delta}} +\|\p_x\omega_1\|_{H^{3/2-\delta}}).
  \end{align*}
  Similarly, for the last term, we may use the structure of the nonlinearity to
  estimate by
  \begin{align}
    \label{eq:30}
    \|\nabla_\tau^\bot \Delta_\tau^{-1} \omega_1(\tau, x,y)\|_{L^\infty} \|\dot \omega_2\|_{L^2}
  \end{align}
  and
  \begin{align}
    \label{eq:31}
    \|\p_x \Delta_\tau^{-1} \omega_1(\tau, x,y)\|_{L^2} \| \nabla \dot \omega_2\|_{L^\infty},
  \end{align}
  since the scalar products involves only pairs of $(\p_y-tU'(y)\p_x, \p_x)$ derivatives.
  In both \eqref{eq:30} and \eqref{eq:31}, the $L^2$ norm decays with an algebraic rate, while the
  $L^\infty$ norms are either bounded or grows at most logarithmically with time.

  Hence, it only remains to consider the $\dot S$ contribution.
  We recall that by the structure of the equation and assuming that $U'' \in
  L^\infty$,  estimating $\dot S$ reduces to estimating
  \begin{align*}
   & \quad  \| \p_x \Delta_\tau^{-1} \left( \nabla_\tau^\bot \Delta_\tau^{-1} \omega_1(\tau, x,y) \cdot \frac{1}{U(y)}\p_x^{-1} \nabla_\tau \omega_2(\tau, x+\tau U(y),y) \right) \|_{L^2} \\
    & \leq \| \p_x \Delta_\tau^{-1} \omega_1(\tau, x,y)\|_{L^2} \|\nabla_\tau \omega_2(\tau, x+\tau U(y),y)\|_{L^\infty} \\
    & \leq t^{-3/2+\delta} \|\p_x \omega_1\|_{H^{3/2-\delta}} \log(t). 
  \end{align*}
  Here, we again used the structure of the nonlinearity as $\nabla_\tau^\bot
  \cdot ( a \nabla_\tau b)$ in order to absorb $\nabla_\tau^\bot$ into $\Delta_\tau^{-1}$.
  \\
  
  \underline{Ad IV:} 
  In view of the leading term in the case of Proposition
  \ref{prop:couetteconsistency}, we claim that $\|IV\|_{H^{-1}}$ is uniformly
  bounded in time. Here, compared to Section \ref{sec:couette} additional
  challenges are given by the logarithmic growth bound on $\|\nabla
  \omega_1\|_{L^\infty}$ and the time-dependence via $S(\cdot,\cdot)$.
  
  For this purpose, we rephrase the integral as
  \begin{align}
    & \quad \int_0^t S(t,\tau) (\nabla^\bot (\Delta^{-1}\omega_2)(\tau,x+\tau U(y), y)) \cdot \nabla \omega_1(\tau,x,y)) d\tau \\
    &= S(t,\tau) \left(\nabla^\bot \frac{1}{U(y)}\p_x^{-1}(\Delta^{-1}\omega_2)(\tau,x+\tau U(y), y)) \cdot \nabla \omega_1(\tau,x,y)\right) \bigg|_{\tau=0}^t \\
    & \quad - \int_0^t \dot S(t,\tau) \left(\nabla^\bot \frac{1}{U(y)}\p_x^{-1} (\Delta^{-1}\omega_2)(\tau,x+\tau U(y), y)) \cdot \nabla \omega_1(\tau,x,y)\right) d\tau \\
    & \quad - \int_0^t S(t,\tau) (\nabla^\bot \frac{1}{U(y)}\p_x^{-1} (\Delta^{-1}\dot \omega_2)(\tau,x+\tau U(y), y)) \cdot \nabla \omega_1(\tau,x,y)) d\tau \\
    & \quad - \int_0^t S(t,\tau) (\nabla^\bot \frac{1}{U(y)}\p_x^{-1} (\Delta^{-1}\omega_2)(\tau,x+\tau U(y), y)) \cdot \nabla \dot \omega_1(\tau,x,y)) d\tau.
  \end{align}
  We note that the first term, again using the structure of the nonlinearity, can
  be estimated in $H^{-1}$ by
  \begin{align*}
    \|\frac{1}{U(y)}\p_x^{-1}(\Delta^{-1}\omega_2)(\tau,x+\tau U(y), y)  \nabla \omega_1(\tau,x,y)\|_{L^2}.
  \end{align*}
  We  observe that control by
  \begin{align*}
     \|\frac{1}{U(y)}\p_x^{-1}(\Delta^{-1}\omega_2)(\tau,x+\tau U(y), y)\|_{L^\infty} \|\omega_1(\tau,x,y)\|_{H^1}
  \end{align*}
  yields a uniform bound in $H^{-1}$.
  Concerning asymptotic stability of this term, we note that $S(\cdot,\cdot)$ is (strongly)
  convergent in $H^{-1}$ and it thus suffices to consider
  \begin{align*}
    \left(\nabla^\bot \frac{1}{U(y)}\p_x^{-1}(\Delta^{-1}\omega_2)(\tau,x,y) \cdot \nabla \omega_1(\tau,x+\tau U(y),y)\right) \bigg|_{\tau=0}^t,
  \end{align*}
  where we switched to Eulerian coordinates.
  Using the null structure of the nonlinearity and that the first factor is
  bounded in $W^{1,\infty}$ by the Sobolev embedding, we obtain weak convergence
  via the duality formulation of $H^{-1}$ and the oscillation of $\omega_1(\tau,
  x+\tau U(y),y)$.

  For the second term, we observe that by the properties of $\dot
  S(\cdot,\cdot)$ we need to estimate 
  \begin{align*}
   \int_0^t \|\Delta_{\tau}^{-1} \left(\nabla^\bot \frac{1}{U(y)}\p_x^{-1} (\Delta^{-1}\omega_2)(\tau,x+\tau U(y), y)) \cdot \nabla \omega_1(\tau,x,y)\right)\|_{H^{-1}} d\tau.
  \end{align*}
  We remark that the operator norm of $\Delta_\tau^{-1}: L^2 \mapsto H^{-1}$ is
  bounded by $C (1+\tau)^{-1}$, which yields a bound $\|IV\|_{H^{-1}}\leq
  C \log(2+t)$.
  In order to improve this estimate we first consider the case when $\omega_2$ is restricted to
  Fourier modes
  \begin{align*}
   \Omega_1(t):=\{(k,\eta): |\eta|\geq c t\}
  \end{align*}
  with respect to $x, z=U(y)$  with a small constant $c>0$ to be fixed later.
  Then it holds that
  \begin{align*}
    \|\Delta^{-1}P_{\Omega_1(t)}\omega_2\|_{L^2} \leq C t^{-2} \|P_{\Omega_1(t)}\omega_2\|_{L^2}.
  \end{align*}
  In this case, we control
  \begin{align*}
    & \quad \|\Delta_{\tau}^{-1} \left(\nabla^\bot_\tau \frac{1}{U(y)}\p_x^{-1} (\Delta^{-1}P_{\Omega_1}\omega_2)(\tau,x+\tau U(y), y)) \cdot \nabla_\tau \omega_1(\tau,x,y)\right)\|_{L^2} \\
   &\leq \| \Delta^{-1}P_{\Omega_1}\omega_2\|_{L^2} \|\omega_1(\tau,x,y)\|_{L^\infty} \leq C t^{-2}.
  \end{align*}
  Similarly, if we restrict $\omega_1$ to $\Omega_1$, we may estimate
  \begin{align*}
    & \quad \|\Delta_{\tau}^{-1} \left(\nabla^\bot_\tau \frac{1}{U(y)}\p_x^{-1} (\Delta^{-1}\omega_2)(\tau,x+\tau U(y), y)) \cdot \nabla_\tau P_{\Omega_1}\omega_1(\tau,x,y)\right)\|_{L^2}
    \\ & \leq \|\nabla^\bot_\tau \frac{1}{U(y)}\p_x^{-1} (\Delta^{-1}\omega_2)(\tau,x+\tau U(y), y)\|_{L^\infty} \|P_{\Omega_1}\omega_1\|_{L^2} \\
    & \leq C t^{-3/2+\delta} \|\omega_1\|_{H^{3/2-\delta}}.
  \end{align*}
  Hence, it remains to estimate the case when both $\omega_1(\tau,x,y)$ and
  $\omega_2(\tau,x,y)$ are projected onto the complement of $\Omega_1(\tau)$.
  Here, we note that, due to the shear in $x$ and the assumed vanishing
  $x$ average of $\omega_2$,
  \begin{align*}
    \p_x^{-1} (\Delta^{-1}P_{\Omega_1^c}\omega_2)(\tau,x+\tau U(y), y))
  \end{align*}
  is highly oscillatory and close to resonant, i.e. $|\eta-kt|\leq ct$.
  However, since $\omega_2$ also vanishes in $x$, this implies that the product
  is localized at frequencies $(k+l, \eta+\xi)$ such that 
  \begin{align*}
    |(\eta+\xi- (k+l)t)| \geq |l|t - 2ct \geq (1-2c)t.
  \end{align*}
  Hence, the product is non-resonant in this case, and we may estimate
  \begin{align*}
    & \quad \|\Delta_{\tau}^{-1} \left(\nabla^\bot_\tau \frac{1}{U(y)}\p_x^{-1} (\Delta^{-1}P_{\Omega_1^c}\omega_2)(\tau,x+\tau U(y), y)) \cdot \nabla_\tau P_{\Omega_1^c}\omega_1(\tau,x,y)\right)\|_{H^{-1}} \\
    & \leq t^{-2} \|\nabla^\bot \frac{1}{U(y)}\p_x^{-1} (\Delta^{-1}P_{\Omega_1^c}\omega_2)(\tau,x+\tau U(y), y)\cdot \nabla P_{\Omega_1^c}  \omega_1(\tau,x,y)\|_{H^{-1}} \\
   & \leq t^{-2} \|\frac{1}{U(y)}\p_x^{-1} (\Delta^{-1}P_{\Omega_1^c}\omega_2)(\tau,x+\tau U(y), y)\|_{L^2} \|\nabla P_{\Omega_1^c}\omega_1\|_{L^\infty}.
  \end{align*}

  Finally, let us consider the last term.  We note that an estimate by
  \begin{align*}
    & \quad \|(\nabla^\bot_\tau \frac{1}{U(y)}\p_x^{-1} (\Delta^{-1}\omega_2)(\tau,x+\tau U(y), y)) \cdot \nabla_\tau \dot \omega_1(\tau,x,y))\|_{H^{-1}} \\
    &\leq \|\nabla^\bot_\tau \frac{1}{U(y)}\p_x^{-1} (\Delta^{-1}\omega_2)(\tau,x+\tau U(y), y)\|_{L^\infty} \| \nabla_\tau \dot \omega_1(\tau,x,y)\|{L^2} \leq  C \tau^{-1},
  \end{align*}
  is again just barely insufficient to establish boundedness of the $\tau$
  integral and that the estimate of the last $L^2$ norm can not be expected to
  be better than $\tau^{-1}$ even if $\omega_1$ were arbitrarily regular.
  However, we note that compared to $IV$ the structure of this term is better in
  that $\dot \omega_1$ exhibits additional decay.
  We hence, repeat our integration by parts argument to obtain
  \begin{align*}
    & \quad S(t,\tau) (\nabla^\bot \frac{1}{U(y)^2}\p_x^{-2} (\Delta^{-1}\omega_2)(\tau,x+\tau U(y), y) \cdot \nabla \dot \omega_1(\tau,x,y)) d\tau \bigg|_{\tau=0}^t \\
   & - \int_0^t \dot S(t,\tau) (\nabla^\bot \frac{1}{U(y)^2}\p_x^{-2} (\Delta^{-1}\omega_2)(\tau,x+\tau U(y), y) \cdot \nabla \dot \omega_1(\tau,x,y)) d\tau \\
   & - \int_0^t S(t,\tau) (\nabla^\bot \frac{1}{U(y)^2}\p_x^{-2} (\Delta^{-1}\dot \omega_2)(\tau,x+\tau U(y), y) \cdot \nabla \dot \omega_1(\tau,x,y)) d\tau \\
   & - \int_0^t S(t,\tau) (\nabla^\bot \frac{1}{U(y)^2}\p_x^{-2} (\Delta^{-1}\omega_2)(\tau,x+\tau U(y), y) \cdot \nabla \ddot \omega_1(\tau,x,y)) d\tau.
  \end{align*}
The first three terms can then be estimated as above, while for the last term,
we may control by 
\begin{align*}
 & \quad  \| (\nabla^\bot_\tau \frac{1}{U(y)^2}\p_x^{-2} (\Delta^{-1}\omega_2)(\tau,x+\tau U(y), y)\|_{L^\infty} \|\nabla_\tau \ddot \omega_1(\tau,x,y)\|_{L^2} \\
  & \leq C \tau^{-3/2+\delta} (\|\omega_1\|_{H^{3/2-\delta}} +\|\p_x\omega_1\|_{H^{3/2-\delta}}).
\end{align*}
In summary, we thus obtain that also this final term is integrable and hence
$\|IV\|_{H^{-1}}\leq C$, as claimed.
\end{proof}

\section{The Forced Navier-Stokes Problem}
\label{sec:simple}

We are interested in the long-time asymptotic behavior of the forced Navier-Stokes equations on the infinite periodic channel $\T \times \R$ near Couette flow
\begin{align}
  \label{eq:21}
  \begin{split}
  \dt \omega + y \p_{x} \omega - \nu \Delta \omega &= f - v \cdot \nabla \omega,\\
  v &= \nabla^{\bot} \phi, \\
 \Delta \phi &=\omega,
  \end{split}
\end{align}
where $\omega$ denotes the nonlinear perturbation to Couette flow and $v$ the
associated perturbation of the velocity field.
Here, we will mostly focus on forcing analogous to the stationary and resonant
cases of Section \ref{sec:couette}.

\subsection{The Linearized Problem}
As a simple introductory model let us consider the linearized equation
\begin{align}
  \label{eq:1}
  \begin{split}
  \dt \omega + y \p_{x} \omega - \nu \Delta \omega &= f,\\
  v &= \nabla^{\bot} \phi, \\
 \Delta \phi &=\omega.
  \end{split}
\end{align}
For an in-depth discussion of the unforced linear and nonlinear case, the
interested reader is referred to \cite{bedrossian2016enhanced}.

We note that like the unforced linear problem allows for an explicit solution by a
Fourier multiplier. 
\begin{lem}
\label{lem:NDuhamel}  
  Let $\omega_0 \in L^2(\T \times I)$ and $f \in C(\R_+, L^2(\T \times I))$.
  Then the explicit solution $\omega$ of \eqref{eq:1} is given by
  \begin{align*}
    \tilde{\omega}(t,k,\eta+kt) &= \exp \left( -\nu \int_0^t k^2+(\eta-k\tau)^2 d\tau \right) \tilde{\omega}_0(k,\eta) \\
                                & \quad + \int_0^t \exp \left( -\nu \int_\tau^{t} k^2+(\eta-ks)^2 ds \right) \tilde{f}(\tau,k,\eta+k\tau) d\tau,
  \end{align*}
\end{lem}
In particular, we note for $k \neq 0$,
\begin{align*}
  \nu \int_0^t k^2+(\eta-k\tau)^2 d\tau \geq \frac{1}{3} k^2\tau^3
\end{align*}
and thus the solution of the unforced problem decays with rate $\exp(-C\nu
t^3)$, which is much faster than the at first expected exponential decay rate of
the heat equation. One says that this problem exhibits \emph{enhanced
  dissipation}.
\begin{proof}
  We note that $W(t,x,y):=\omega(t,x+ty,y)$ satisfies the equation
  \begin{align*}
    \dt W - \nu (\p_x^2+(\p_y-t\p_x)^2) W= f(t,x+ty,y) \\
    \leadsto \dt \tilde{W}(t,k,\eta) + \nu (k^2+(\eta-kt)^2) \tilde{W}(t,k,\eta)= \tilde{f}(t,k,\eta+kt).
  \end{align*}
  The claimed solution formula hence follows by an application of Duhamel's
  principle to this family of inhomogeneous ordinary differential equations. 
\end{proof}

As in Section \ref{sec:couette}, we consider as prototypical cases
$f(t,x,y)=f_0(x,y)$ being \emph{stationary} and $f(t,x,y)=f_0(x-ty,y)$ being
\emph{resonant}.
\begin{cor}
  \label{cor:linearizedcases}
  Let $\omega_0, f_0 \in L^2(\T \times \R)$ with $\int \omega_0 dx= 0 =\int f_0 dx$ and let $f(t,x,y)=f_0(x-ty,y)$, then
  the solution of \eqref{eq:1} with initial datum $\omega_0$ is given by
  \begin{align*}
    \tilde{\omega}(t,k,\eta+kt) &= \exp \left( -\nu \int_0^t k^2+(\eta-k\tau)^2 d\tau \right) \tilde{\omega}_0(k,\eta) \\
                                & \quad + \tilde{f}_0(k,\eta) \int_0^t \exp \left( -\nu \int_\tau^t k^2+(\eta-ks)^2 ds \right) d\tau \\
                                & =: \exp \left( -\nu \int_0^t k^2+(\eta-k\tau)^2 d\tau \right) \tilde{\omega}_0(k,\eta) \\
                                & \quad + b(t,k, \eta+kt)  \tilde{f}_0(k,\eta).
  \end{align*}
  Here, $b(t,k,\eta)\geq 0$ is monotonically increasing in $t$, bounded and
  $b_\infty(t):=\lim\limits_{t\rightarrow \infty} b(t,k,\eta)$ only decays with an
  algebraic rate $(k^2+\eta^2)^{-1}$.
  In particular, it follows that
  \begin{itemize}
  \item The evolution (of $\omega(t,x,y)$ or $\omega(t,x+ty,y)$) is
    stable in $H^s, s \leq 0$ and the solution asymptotically
    converges to zero. 
    However, unless $f_0$ is trivial, the decay rate is only algebraic. 
  \item The associated velocity field $v(t)=\nabla^\bot
    \Delta^{-1}\omega(t)$ is stable in $L^2$ and converges to zero in $L^2$ as $t
    \rightarrow \infty$.
  \end{itemize}

  If we instead consider the case where $f(t,x,y)=f_0(x,y)$ is stationary with
  $\int f_0 dx =0$, then there exists a stationary solution $g$ of \eqref{eq:1}
  and any solution $\omega$ with initial data $\omega_0 \in L^2$, $\int \omega_0
  dx=0$ satisfies
  \begin{align*}
    \|\omega(t)-g\|_{L^2}\leq \exp(-\frac{\nu}{3} t^3)\|\omega_0-g\|_{L^2}.
  \end{align*}
\end{cor}

\begin{proof}
  Concerning the resonant case, we note that
  $\tilde{f}(\tau,k,\eta+k\tau)=\tilde{f}_0(k,\eta)$ does not depend on time and
  hence obtain the claimed solution formula as a corollary of Lemma
  \ref{lem:NDuhamel}.
  We further note that using several changes of variables
  \begin{align*}
    b(t,k,\eta) &= \int_0^t \exp \left( -\nu \int_\tau^t k^2+(\eta+k(t-s))^2 ds \right) d\tau \\
    &=  \int_0^t \exp \left( -\nu \int_0^{t-\tau} k^2+(\eta+k\sigma)^2 d\sigma \right) d\tau \\
                &= \int_0^t \exp \left( -\nu \int_0^{\xi} k^2+(\eta+k\sigma)^2 d\sigma \right) d\xi \\
    &= \int_0^t \exp \left( -\nu k^2 \xi  - \nu \frac{(\eta+k\xi)^3 - \eta^3}{3k} \right) d\xi.
  \end{align*}
  In particular, we note that $b(t,k,\eta)$ is non-negative and monotonically
   increasing in $t$ and uniformly bounded by $\|\exp(-\nu k^2
   \xi)\|_{L^1}=\frac{1}{\nu k^2}$ and thus there exists
   $ b_\infty(k,\eta)= \lim_{t \rightarrow \infty}b(t,k,\eta)$.
  We claim that for $|\eta| \geq 1$, it holds that
  \begin{align}
    \label{eq:14}
   \frac{c}{k^2+ \eta^2} \leq  b_\infty(k,\eta) \leq \frac{\nu^{-1}}{k^2+\eta^2}.
  \end{align}
  In particular, it follows that for large $t$, $\|b(t,k, \eta)
  \tilde{f}_0(k,\eta-kt)\|_{L^2} \approx \|f_0(x-ty,y)\|_{\dot H^{-2}}$
  generally only decays with algebraic rates depending on the regularity of
  $f_0$.
  We remark that we have seen in Section \ref{sec:couette}, that for $\nu
  \downarrow 0$  $b(t,k,\eta)=t$ is unbounded in $t$ but independent of $\nu, k,
  \eta$.
  \\
  
  It remains to establish \eqref{eq:14}.
  Here, we note that for $\xi<\frac{|\eta|}{10|k|}$ it holds that
  \begin{align*}
    \frac{(\eta+k\xi)^3 - \eta^3}{3k \xi} \xi \approx \eta^2 \xi,
  \end{align*}
  since the difference quotient approximates the derivative.
  Hence, on that interval
  \begin{align*}
    & \quad \int_0^{\frac{|\eta|}{10|k|}} \exp \left( -\nu k^2 \xi  - \nu \frac{(\eta+k\xi)^3 - \eta^3}{3k} \right) d\xi \\
    &\approx -\frac{1}{k^2 + \eta^2}  \nu^{-1}\exp(-\nu(k^2+\eta^2)\xi)\bigg|_{\xi=0}^{\frac{|\eta|}{10|k|}} .
  \end{align*}
  We remark that for $k, \eta$ small compared to $\nu^{-1}$ the $\nu^{-1}\exp(\cdot)$ difference can be
  estimated by the difference quotient in $0$ and hence can be bounded below by a
  uniform constant.
  If instead $(k^2+\eta^2)\frac{|\eta|}{10|k|}$ is of size $\nu^{-1}$ or larger,
  the difference of the exponentials is bounded below and hence the integral on
  this interval is comparable to $\frac{\nu^{-1}}{k^2+\eta^2}$.
  On the complement, $\xi \geq \frac{|\eta|}{10|k|}$, we simply estimate
  \begin{align*}
    0 & \leq \int_{\xi \geq \frac{|\eta|}{10|k|} } \exp \left( -\nu k^2 \xi  - \nu \frac{(\eta+k\xi)^3 - \eta^3}{3k} \right) d\xi \\
    & \leq \int_{\xi \geq \frac{|\eta|}{10|k|} } \exp \left( -\nu k^2 \xi\right) d\xi \leq \frac{\nu^{-1}}{k^2} \exp(-\nu \frac{k\eta}{10}).
  \end{align*}
  This concludes our proof for the resonant case.
  \\
  
  In the stationary case, we claim that
  \begin{align*}
   L g := y \p_x g + \nu \Delta g =f_0,
  \end{align*}
  possesses a unique weak solution $g \in H^1(\T \times \R)$ and that
  $\|g\|_{H^1}\leq 2 \nu^{-1} \|f_0\|_{L^2}$.
  Taking $g$ as a particular solution, the claim then follows immediately by the
  enhanced dissipation acting on the particular (and thus unforced) solution
  with initial data $\omega_0-g$.
  It remains to prove the existence and uniqueness of solutions.
  As $y \p_x$ is an unbounded operator on $H^1(\T \times \R)$ we argue via a
  family of auxiliary problems.
  For each $R>0$ we define $g_R \in H^1_0(\T \times (-R,R))$ with $\int g_R dx =0$ to be the unique solution of
  \begin{align*}
    y \p_x g + \nu \Delta g &=f_0 \text{ in } \T \times (-R,R), \\
    g|_{y=-R,R}&=0.
  \end{align*}
  We note that since $y$ is bounded on $\T \times (-R,R)$,
  \begin{align*}
    B_R(h_1,h_2)= \nu \langle \nabla h_1, \nabla h_2 \rangle_{L^2} + \langle h_1, y \p_x h_2 \rangle_{L^2}
  \end{align*}
  is a bounded bilinear form on $H^1_0(\T \times (-R,R))$ and that
  \begin{align*}
    B_R(h,h)= \nu \|\nabla h\|_{L^2}^2 + \int y \frac{1}{2}\p_x|h|^2 = \nu \|\nabla h\|_{L^2}^2 \geq \frac{\nu}{2} \|h\|_{H^1}^2
  \end{align*}
  is coercive, where we used that $\int g_R dx=0$.
  Hence, $g_R$ exists by the Lax-Milgram theorem and by testing the equation we
  further obtain that
  \begin{align*}
    \|g_R\|_{H^1} \leq \nu^{-1}\|f_0\|_{L^2}.
  \end{align*}
  The bounded sequence $(g_R)_R$ hence posses a weak limit $g \in H^1(\T \times
  \R)$ along some subsequence $R_j\rightarrow \infty$.
  We claim that $g$ is a weak solution.
  Indeed, let $\phi \in C_c^\infty(\T \times \R)$ be a test function whose support is
  contained in $\T \times [-L,L]$. Then for any $R>L$ it holds that
  \begin{align*}
    \nu \int \nabla g_R \cdot \nabla \phi + \int y \p_x\phi g_R = \int \phi f,
  \end{align*}
  which implies 
  \begin{align*}
    \nu \int \nabla g \cdot \nabla \phi + \int y \p_x\phi g = \int \phi f
  \end{align*}
  by letting $R=R_j\rightarrow \infty$.
  Concerning the uniqueness, taking differences of two solutions $g_1,g_2$ it
  suffices to show that the problem with $f_0$ only has a trivial solution.
  Suppose not and let $g \in H^1(\T \times \R)$ be non-trivial such that
  \begin{align*}
     y \p_x g + \nu \Delta g =0.
  \end{align*}
  Let $\chi(y)$ be a standard smooth, symmetrically decreasing cut-off function and
  $\chi_R(y)=\chi(\frac{y}{R})$. Then $\chi_R(y)g \in H^1$ is compactly
  supported test function and
  \begin{align*}
    \int g y \p_x \chi_R(y)g + \nu \int \nabla g \nabla (\chi_R(y)g) = 0 \\
    \Leftrightarrow 0 + \int \chi_R(y)|\nabla g|^2 + \frac{1}{R} \int \chi'(\frac{y}{R}) |g|^2=0.
  \end{align*}
  Here, the first integral vanished as a total derivative in $x$, the second
  integral converges to $\|\nabla g\|_{L^2}^2$ by monotone or dominated
  convergence and the last integral is bounded by $\frac{C}{R}\|g\|_{L^2}^2$ and
  hence tends to zero.
  Thus $\|\nabla g\|_{L^2}^2=0$ and hence $g=0$, which concludes the proof.
\end{proof}

\subsection{The Nonlinear Problem}
\label{sec:nonlinearN}
In the following we discuss the nonlinear forced Navier-Stokes equations
\begin{align}
  \label{eq:1}
  \begin{split}
  \dt \omega + y \p_{x} \omega - \nu \Delta \omega &= f,\\
  v &= \nabla^{\bot} \phi, \\
 \Delta \phi &=\omega.
  \end{split}
\end{align}
We remark that in the linearized problem the evolution of the $x$-average (and
hence the underlying shear) and its $L^2$-orthogonal complement decoupled and that there we could hence without
loss of generality restrict to study the case with vanishing $x$-average.
In the present nonlinear case, the forcing and the nonlinearity introduce a
coupling which, while damped by dissipation, poses technical challenges.
In view of Section \ref{sec:couette} we thus opt for a somewhat simpler setting
where $f$ is chosen to counteract changes to the shear profile due to $\langle v
\cdot \nabla \omega\rangle_x$.
That is, with slight change of notation we consider the problem
\begin{align}
    \begin{split}
  \dt \omega + y \p_{x} \omega - \nu \Delta \omega &= f - (v \cdot \nabla \omega)_{\neq},\\
  v &= \nabla^{\bot} \phi, \\
 \Delta \phi &=\omega,
  \end{split}
\end{align}
where $(\cdot)_{\neq}$ denotes the $L^{2}$ projection on non-zero frequencies
with respect to $x$ and $f=f_{\neq}$.

Given any sufficiently small given stationary forcing $f(t,x,y)=f_0(x,y)$, these
equations then admit a stationary solution.
\begin{prop}
  \label{prop:fixedpoint}
  Suppose that $\|f\|_{H^1}\leq C_1 \nu^2 $ with $C_1\leq \frac{1}{40}$, then there exists $g \in H^{2}$ a weak
  solution of
  \begin{align}
    \label{eq:15}
    y \p_{x} g +\nu \Delta g+ (v[g] \cdot \nabla g)_{\neq 0} =f,
  \end{align}
  with $\|g\|_{H^1}\leq 2C_1 \nu$ and $\|g\|_{H^2}\leq 2C_1$.
\end{prop}

\begin{proof}
  We argue by fixed point iteration.
  Let $L=y\p_{x}+\nu \Delta$ be the linear operator on the space
  \begin{align*}
    X:=\{ u \in H^1(\T \times \R): \langle u \rangle_x=0\}.
  \end{align*}
  We then note that $\nu \Delta$ is a symmetric operator on this space with
  respect to the $L^2$ inner product, possesses a spectral gap and that $y\p_x$
  is anti-symmetric.
  We recall that well-posedness of $L^{-1}$ and its mapping properties have been
  studied in the proof of Corollary \ref{cor:linearizedcases}. 

  Rephrasing our equation as
  \begin{align}
    \label{eq:fixedp}
    L g &= f- (v[g]\cdot \nabla g)_{\neq},
  \end{align}
  we intend to show that there exists a unique fixed point $g$ of the mapping
  \begin{align*}
    \Psi: g & \mapsto L^{-1}(f- (v[g] \nabla g)_{\neq}).
  \end{align*}
  In particular, we show that under the conditions of the proposition, there
  exists $c=c(C_1)>0$ such that $\Psi:B_{c\nu}(0)\rightarrow B_{c\nu}(0)$ and that $\Psi$ is a contraction.
  The existence of $g$ then follows by the Banach fixed point theorem.
  \\
  
  Let thus $c>0$ to be fixed later and $g \in X$ with $\|g\|_{H^1}\leq c \nu$.
  Then $v[g] \in L^\infty$ with $\|v_g\|_{L^\infty}\leq 10 c \nu$ by the Sobolev embedding and $\|\nabla g\|_{L^2} \leq
  c \nu$, thus
  \begin{align*}
    \|f- (v[g] \nabla g)_{\neq}\|_{L^2}\leq C_1 \nu^2 + 10 c^2 \nu^2.
  \end{align*}
  Let now $u=\Psi(g)$ and assume for the moment that $u$ is decaying
  sufficiently quickly such that $\int u y \p_x u=0$.
  Formally testing with $-u$ it follows that
   \begin{align*}
     \nu \|\nabla u \|_{L^2}^2 &= - \int u (f- (v[g] \nabla g)_{\neq}) \leq \|u\|_{L^2} \|f- (v[g] \nabla g)_{\neq}\|_{L^2}\\
     \Rightarrow \|\nabla u \|_{L^2} &\leq \frac{1}{\nu} \|f- (v[g] \nabla g)_{\neq}\|_{L^2}\leq  C_1 \nu+ 10 c^2 \nu,
   \end{align*}
  where we used that $\|u\|_{L^2}\leq \|\p_x u\|_{L^2}\leq \|\nabla u\|_{L^2}$
  due to the vanishing $x$ average.
  Hence, $\Psi$ maps $B_{c\nu}(0)$ into itself provided $C_1 + 10c^2 \leq c$,
  which is satisfied by $c=2C_1\leq \frac{1}{20}$.
  We remark that, as in Corollary \ref{cor:linearizedcases}, $y\p_x$ is an
  unbounded operator on $H^1$ and that hence, we may only formally test with
  $-u$. In order to make the preceding argument rigorous, we may thus again consider a standard cut-off function
  $\Psi_R(y)=\Psi_1(\frac{y}{R})$ on $\R$ and test with $-\Psi_R(y)u$ and
  subsequently take the limit $R \rightarrow \infty$.
 \\

  Concerning the contraction property, let $g_1,g_2 \in X$ with $\|g_1\|_{H^1},
  \|g_2\|_{H^1}\leq c \nu$. Then by the previous estimate and the quadratic
  structure of the nonlinearity:
  \begin{align*}
    \|\Psi(g_1)-\Psi(g_2)\|_{H^1}^2 &\leq \frac{1}{\nu^2} \|(v[g_1] \nabla g_1)_{\neq}- (v[g_1] \nabla g_1)_{\neq}\|_{L^2}^2 \\
    &\leq \frac{C}{\nu^2} \|g_{1}+g_{2}\|_{H^{1}}^2 \| g_{1}-g_{2}\|_{H^1}^2 \leq C (2c)^2  \|g_{1}-g_{2}\|_{H^{1}}^2,
  \end{align*}
  where $C \leq 10$ is the constant of the Sobolev embedding. By our choice of
  $c$, this is hence a contraction.
  Thus, a solution $g \in X$ with $\|g\|_{H^1}\leq c \nu$, $c=2C_1$ exists by
  the Banach fixed point theorem.
  In order to establish improved regularity, we formally test the fixed point
  equation with $\nabla g$ to obtain that
  \begin{align*}
    \nu \|\Delta g\|_{L^2}^2 + \int \Delta g y \p_x g \leq \|\Delta g\|_{L^2} \|f - v_g \cdot \nabla g)_{\neq}\|_{L^2}.
  \end{align*}
  Furthermore, integrating by parts we may estimate
  \begin{align*}
    \int \Delta g y \p_x g = - \int y \p_x \frac{|\nabla g|^2}{2} + \int \p_yg \p_x g = 0  + \int \p_yg \p_x g \leq \|g\|_{H^1}^2,
  \end{align*}
  which is controlled by the previous estimate.
  
\end{proof}

We have thus constructed one particular stationary solution of the nonlinear,
viscous, forced problem. We in particular stress that due to the forcing the
stationary state $g$ is not stable in $H^s,s>0$ with respect to Lagrangian
variables $(x-tU(y),y)$, but only with respect to Eulerian coordinates, and not damped further.
However, as we show in the following it is a stable asymptotic state and small
perturbation are damped to $g$.
\begin{prop}
  Let $f(t,x,y)=f_0(x,y)$ satisfy the assumptions of Proposition
  \ref{prop:fixedpoint} and let $g$ be the associated stationary solution with
  $\|\nabla g\|_{L^\infty}\leq c \nu$.
  Suppose that $\omega$ is a global solution of \eqref{eq:21} with $\langle \omega(t) \rangle_x=0$.
  Then $\omega^\star$ satisfies the equation
  \begin{align}
    \label{eq:16}
    \dt \omega^{\star} + y \p_{x} \omega^{\star} - \nu \Delta \omega^{\star} + \left(v^{\star} \cdot \nabla \omega^{\star} - v[g]\cdot \nabla \omega^{\star}- v^{\star}\cdot \nabla g\right)_{\neq}=0,
  \end{align}
  where $()_{\neq}$ denotes the $L^2$ projection onto functions with vanishing $x$-average.
  Furthermore, $\omega^{\star}$ is nonlinearly stable in $L^{2}$ and decays at an (at
    least) exponential decay rate $\exp(-C \nu t)$.
\end{prop}

\begin{proof}
  We note that the equation immediately follows by subtracting \eqref{eq:15}
  from \eqref{eq:21}.
  The $L^{2}$ stability then follows from a formal energy argument.
  That is, testing with $\omega^{\star}$ we obtain that
  \begin{align*}
    \frac{d}{dt}\|\omega^{\star}\|_{L^{2}}/2 + 0 + \nu \|\nabla \omega^{\star}\|_{L^{2}}^{2} +0 -0 - \langle \omega^{\star}, v^{\star}\cdot \nabla g \rangle=0.
  \end{align*}
  Note that here we use the assumption that $\omega$ and $\omega^\star$ is
  sufficiently regular such that $v^{\star} \cdot \nabla \omega^{\star}$ is
  well-defined and that we may integrate by parts.
  
  Using that $\|\nabla g\|_{L^{\infty}}\leq c \nu$ for a small constant $c$, then
  we can estimate
  \begin{align*}
    \langle \omega^{\star}, v^{\star}\cdot \nabla g \rangle &= \langle \omega^{\star}, \nabla \phi^{\star} \cdot \nabla^{\bot}g \rangle\\
    = - \langle \nabla \omega^{\star}, \phi^{\star} \nabla^{\bot}g \rangle &\leq c \nu \|\nabla \omega^{\star}\|_{L^{2}}\|\phi^{\star}\|_{L^{2}}.
  \end{align*}
  Since $\omega$ and hence $\omega^\star$ and $\phi^\star$ possess a
  vanishing average in $x$ and hence can be controlled using the Poincar\'e
  inequality.
  Thus, it follows that
  \begin{align*}
    \frac{d}{dt}\|\omega^{\star}\|_{L^{2}}/2 \leq -\frac{\nu}{2} \|\nabla \omega^{\star}\|_{L^{2}}^{2}\leq -\frac{\nu}{2} \|\omega^{\star}\|_{L^{2}}^{2} ,
  \end{align*}
  which implies (at least) exponential decay with rate $\exp(-\frac{\nu}{2}t)$.
\end{proof}

\bibliographystyle{alpha} \bibliography{citations}

\end{document}